\documentclass[11pt]{article}
\usepackage{amsthm, amsmath, amssymb, amsfonts, url, booktabs, tikz, setspace, fancyhdr, amsbsy}
\usepackage[margin = 0.8in]{geometry}
\usepackage{esint}
\usepackage{verbatim}

\newtheorem{theorem}{Theorem}[section]

\newtheorem{lemma}[theorem]{Lemma}

\theoremstyle{definition}
\newtheorem{definition}[theorem]{Definition}

\theoremstyle{remark}

\newcommand{\norm}[1]{\left\lVert#1\right\rVert}

\newcommand{\R}{\mathbb{R}}
\newcommand{\N}{\mathbb{N}}

\newcommand{\defeq}{\mathrel{\mathop:}=}
\newcommand{\eqdef}{\mathrel{\mathop=}:}

\newcommand{\uu}{\boldsymbol{u}}
\newcommand{\uuu}{\boldsymbol{\tilde{u}}}
\newcommand{\DD}{\boldsymbol{D}}
\newcommand{\BB}{\boldsymbol{B}}
\newcommand{\CC}{\boldsymbol{C}}
\newcommand{\MM}{\boldsymbol{M}}

\newcommand{\vv}{\boldsymbol{v}}
\newcommand{\ppsi}{\boldsymbol{\psi}}
\newcommand{\pphi}{\boldsymbol{\phi}}
\newcommand{\SSS}{\boldsymbol{S}}
\newcommand{\GGG}{\boldsymbol{G}}

\newcommand{\ww}{\boldsymbol{w}}

\newcommand{\dx}{\,\mathrm{d}x}

\newcommand{\dt}{\,\mathrm{d}t}
\newcommand{\ds}{\,\mathrm{d}s}
\newcommand{\dxi}{\,\mathrm{d}\xi}
\newcommand{\DDD}{\overline{\DD}}
\numberwithin{equation}{section}

\def\ocirc#1{\ifmmode\setbox0=\hbox{$#1$}\dimen0=\ht0
    \advance\dimen0 by1pt\rlap{\hbox to\wd0{\hss\raise\dimen0
    \hbox{\hskip.2em$\scriptscriptstyle\circ$}\hss}}#1\else
    {\accent"17 #1}\fi}

\begin{document}

\title{Upper and lower bounds of convergence rates for strong solutions of the generalized Newtonian fluids with non-standard growth conditions}

\author{Jae-Myoung Kim\thanks{Department of Mathematics Education, Andong National University, Andong, 36729, Korea (Republic of). Email: \tt{jmkim02@andong.ac.kr}}
~and ~Seungchan Ko\thanks{Department of Mathematics, The University of Hong Kong, Pokfulam Road, Hong Kong. Email: \tt{scko@maths.hku.hk}}}

\date{~}

\maketitle

~\vspace{-1.5cm}

\begin{abstract}
We consider the motion of an incompressible shear-thickening power-law-like non-Newtonian fluid in $\R^3$ with a variable power-law index. This system of nonlinear partial differential equations arises in mathematical models of electrorheological fluids. The aim of this paper is to investigate the large-time behaviour of the difference $\uu-\uuu$ where $\uu$ is a strong solution of the given equations with the initial data $\uu_0$ and $\uuu$ is the strong solution of the same equations with perturbed initial data $\uu_0+\ww_0$. The initial perturbation $\ww_0$ is not required to be small, but is assumed to satisfy certain decay condition. In particular, we can show that
\[(1+t)^{-\frac{\gamma}{2}}\lesssim \|\uuu(t)-\uu(t)\|_2\lesssim(1+t)^{-\frac{\gamma}{2}},\]
for sufficiently large $t>0$, where $\gamma\in(2,\frac{5}{2})$. The proof is based on the observation that the solution of the linear heat equation describes the asymptotic behaviour of the solutions of the electrorheological fluids well for sufficiently large time $t>0$, and the generalized Fourier splitting method with an iterative argument. Furthermore, it will also be discussed that the argument used in the present paper can improve the previous results for the generalized Newtonian fluids with a constant power-law index.
\end{abstract}

%whose viscosity is significantly changed by the surrounding electro-magnetic field

\noindent{\textbf{Keywords:} Non-Newtonian fluid, variable exponent, electrorheological fluid, optimal convergence rates, upper and lower bounds}

\smallskip

\noindent{\textbf{AMS Classification:} 76A05, 76D05, 35B40}

\begin{section}{Introduction}
In this paper, we aim to investigate the large-time behaviour for solutions of a system of nonlinear partial differential equations (PDEs) describing the rheological response of electrorheological fluids. The electrorheological fluids, which have recently gained increasing attention, are a special type of smart fluids of high technological interest and are  characterized by their ability to change their rheological properties when in presence of an surrounding electromagnetic field: When it is disposed to an electro-magnetic field, the viscosity exhibits a significant change.  For example, there exist some types of electrorheological fluids whose viscosity varies by a factor of $1000$ as a response to the application of an electromagnetic field within 1ms.  This property can be exploited in a wide range of scientific and industrial applications,
 including for instance, shock absorbers, clutches and actuators. 

In the present paper, we shall consider the reduced model for the incompressible electrorheological fluids consisting of the following system of PDEs:
\begin{alignat}{2}
	\partial_t\uu+(\uu\cdot\nabla)\uu-{\rm{div}}\,\SSS(\DD\uu)&=-\nabla \pi+\boldsymbol{f}\qquad &&{\rm{in}}\,			\,Q_T=(0,T)\times\R^d,\label{eq1}\\
	{\rm{div}}\,\uu&=0\qquad &&{\rm{in}}\,\,Q_T=(0,T)\times\R^d,\label{eq2}
\end{alignat}
where the Cauchy stress tensor is of the form
\begin{equation}\label{ST}
	\SSS(\DD\uu)=(1+|\DD\uu|^2)^{\frac{p(t,x)-2}{2}}\DD\uu.
\end{equation}
In fact, the exponent $p(\cdot)$ depends on the magnitude of the electric field $|\boldsymbol{E}|$. Since the electric field itself is a solution to the quasi–static Maxwell equations and is not influenced by the motion of the fluids, we may consider $p(\cdot)$ as a given function and restrict ourselves to the mathematical study of the equations \eqref{eq1}-\eqref{ST}.
In the above system of equations, $\uu:Q_T\rightarrow\R^d$, $p:Q_T\rightarrow\R$, denote the velocity field and pressure respectively, and $\DD\uu$ is the symmetric velocity gradient, i.e. $\DD\uu=\frac{1}{2}(\nabla\uu+(\nabla\uu)^T)$. Here we prescribe the initial condition
\begin{equation}\label{initial}
	\uu(0,x)=\uu_0(x)\quad{\rm{in}}\,\,\R^d.
\end{equation}
Such electrorheological models were studied in \cite{R2000, R2004}, where the mathematical modelling, existence theory and numerical approximation were discussed. 

In this paper, we shall investigate the large-time behaviour of the solutions for the model \eqref{eq1}-\eqref{ST}. Because of its importance in mathematics and physics, the large-time behaviour of viscous incompressible fluids has attracted more attention and has been investigated in many works. For the classical Navier--Stokes equations, there has been an extensive literature; see, e.g., \cite{Schonbek_1, Schonbek_2} where the Fourier splitting method was initially developed, and \cite{kato, ref_1, ref_2, Schonbek_3, Schonbek_4, Zhang_1995, titi, Chen_1991, Han_2015, Han_2016_1, Han_2016_2} for related topics including upper and lower bounds of the decay rates for several types of norms and classes of initial data in various domains.
% In \cite{mag_1, mag_2, mag_3, mag_4}, the authors considered the magnetohydrodynamics equaions which consist of the classical Navier--Stokes equations and Maxwell's equations, and obtained the similar results. 
There also have been a large number of contributions for the non-Newtonian fluid flow model. In \cite{non_1, non_2, Dong, Bae_1999, Dong_2004} examined the decay properties of solutions for the power-law-like non-Newtonian fluid flow models.
%and the similar results for the generalized Navier--Stokes equations coupled with Maxwell's equations were studied in \cite{non_mag_2, non_mag_1}.

On the other hand, for the mathematical study of the algebraic convergence rate for solutions of Navier--Stokes equations under initial perturbations, in \cite{perturb_NS}, the authors showed the algebraic convergence of every solution of the perturbed problem to the solution of the Navier--Stokes equations, within the Morrey spaces framework. Similarly, in \cite{perturb_non}, the authors proved the convergence properites of weak solution to the viscous incompressible generalized Newtonian fluids without the assumption that the initial perturbation is small. As it will be made clear later, our argument used in this paper can improve the result in \cite{perturb_non} so that we can have the similar result with the lower values of $p>0$.

For the mathematical model \eqref{eq1}-\eqref{initial} describing the motion of the incompressible electrorheological fluids,  to the best of our knowledge, the only result for the large-time behaviour of the solutions is $L^2$ decay rates for the strong solution presented in \cite{Ko2022}. As a subsequent study, it is natural and desirable to investigate the asymptotic stability of the solutions of the equations \eqref{eq1}-\eqref{initial}. This paper aims to study the convergence rate for this type of problem. Specifically, the perturbed equations of \eqref{eq1}-\eqref{initial} for the initial perturbation can be written as
\begin{alignat}{2}
	\partial_t\uuu+(\uuu\cdot\nabla)\uuu-{\rm{div}}\,\SSS(\DD\uuu)&=-\nabla {\tilde{\pi}}+\boldsymbol{f}\qquad &&{\rm{in}}		\,\,Q_T,\label{p_eq1}\\
	{\rm{div}}\,\uuu&=0\qquad &&{\rm{in}}\,\,Q_T,\label{p_eq2}\\
	\uuu_0(x)&=\uu(0,x)+\ww(0,x) \qquad&&{\rm{in}}\,\,\R^d.\label{p_initial}
\end{alignat}
The main result of this paper is that for any large initial perturbation satisfying some decay condition, the perturbed solution $\uuu(t,x)$ of \eqref{p_eq1}-\eqref{p_initial} converges to the strong solution $\uu(t,x)$ of the original problem \eqref{eq1}-\eqref{initial} with the optimal upper and lower bounds
\begin{equation}\label{main_result}
C_1(1+t)^{-\frac{\gamma}{2}}\leq\|\uuu(t)-\uu(t)\|_{L^2(\R^3)}\leq C_2(1+t)^{-\frac{\gamma}{2}},
\end{equation}
where $2<\gamma<\frac{5}{2}$. Note that the above convergence result is optimal in the sense that it coincides with the convergence rate of the linear heat equations.

The remaining of the present paper is organized as follows. In Section 2, we discuss some notations and preliminaries which will be used throughout the paper, and present the precise statement of the main theorem. As the proof of the main theorem, in Section 3, we examine the estimate for the upper bound in \eqref{main_result}. Subsequently, in Section 4, we prove the optimal lower bounds for the convergence rate in \eqref{main_result}.
Finally, in Section 5, we present some concluding remarks on the paper.

\end{section}

\begin{section}{Preliminaries and main theorem}
In this section, we first introduce some notations and  discuss preliminaries which will be needed throughout the paper. For two vectors $\boldsymbol{a}$ and $\boldsymbol{b}$, $\boldsymbol{a}\cdot \boldsymbol{b}$ denotes the scalar product; and, similarly, for two tensors $\mathbb{A}$ and
$\mathbb{B}$, $\mathbb{A}:\mathbb{B}$ signifies their scalar product. Throughout the paper, the notation $P\lesssim Q$ means that there exists some constant $C>0$ such that $P\leq CQ$. Also, $C$ denotes a generic positive constant, which may change at each appearance. For $1\leq p\leq\infty$ and $k\in\N$, we mean by $W^{k,p}(\R^d)$ the standard Sobolev space and we denote $H^k(\R^d)=W^{k,2}(\R^d)$. Furthermore, for simplicity, we write $\|\cdot\|_p=\|\cdot\|_{L^p(\R^d)}$ and $\|\cdot\|_{k,p}=\|\cdot\|_{W^{k,p}(\R^d)}$.

 We first recall Korn's inequality (see, for example, Lemma 2.7 in \cite{Cauchy}).
\begin{lemma}
Assume that $1<q<\infty$. Then there exists a positive constant $C>0$ depending on $q$ such that for any $\uu\in W^{1,q}(\R^d)$, we have
\[\|\nabla\uu\|_q\leq C\|\DD\uu\|_q.\]
\end{lemma}
We will also use the classical Gronwall's inequality, which is well-known in the theory of differential equations.
\begin{lemma}\label{gronwall}
For any positive continuous functions $f(t),g(t),h(t)$ with $f'(t)\geq0$ for all $t>0$. if they satisfy the inequality
\[g(t)\leq f(t)+\int^t_0g(s)h(s)\ds,\quad\forall t>0,\]
then we have
\[g(t)\leq f(t)\exp\left(\int^t_0h(s)\ds\right),\quad\forall t>0.\]
\end{lemma}
Furthermore, for a given integrable function $f\in L^1(\R^d)$, we shall denote the Fourier transform of $f$ by
\[\widehat{f}(\xi) = \int_{\R^d}f(x)e^{-2\pi i x\cdot \xi}\dx,\]
and throughout our analysis, we will use the classical Plancherel's theorem frequently.
\begin{lemma}\label{plan}
Suppose that $f\in L^1(\R^d)\cap L^2(\R^d)$. Then we have
\[\int_{\R^3}|f(x)|^2\dx=\int_{\R^3}|\widehat{f}(\xi)|^2\,{\rm{d}}\xi.\]
\end{lemma}

The following decay properties of the heat equation is of independent interest and will be used in our later analysis (see, for instance, Proposition 3 in \cite{titi}).
\begin{lemma}\label{linear_decay}
For any real number $m\in\R$, assume that the initial data $\uu_0\in H^m(\R^d)$ satisfies
\begin{equation}\label{initial_decay}
\int_{\mathbb{S}^{d-1}}|\widehat{\uu}_0(r\omega)|^2\,{\rm{d}}\omega=Cr^{2\gamma-d}+o(r^{2\gamma-d})\quad{\rm{as}}\,\,r\rightarrow0,
\end{equation}
for some constant $C>0$. Then for some positive constants $C_1>0$ and $C_2>0$, the solution of the heat equation
\begin{align*}
\partial_t \uu-\Delta \uu&=0,\\
\uu(0,x)&=\uu_0(x),
\end{align*}
has the following decay properties: for sufficiently large $t>0$,
\[C_1(1+t)^{-\frac{m+\gamma}{2}}\leq\|\nabla^m e^{\Delta t}\uu_0\|_2\leq C_2(1+t)^{-\frac{m+\gamma}{2}}.\]
\end{lemma}
Note that if the solution $\uu$ of the heat equation is divergence-free at $t=0$, $\uu$ will remain divergence free for all $t>0$. This can be proved by taking the divergence operator to the equations and integrating against ${\rm{div}}\,\uu$.

Now, as we deal with the variable power-law index, we need to consider  the variable-exponent Lebesgue space. We write a set of all measurable functions $p:\Omega\rightarrow[1,\infty)$ by $\mathcal{P}(\Omega)$ and we shall call $p\in\mathcal{P}(\Omega)$ a variable exponent. Let $\Omega\subset\R^d$ be an open set and $p\in\mathcal{P}(\Omega)$. We then define $p^-={\rm{ess}}\,\inf_{x\in\Omega}p(x)$ and $p^+={\rm{ess}}\,\sup_{x\in\Omega}p(x).$ Also, for each measurable functions $f:\Omega\rightarrow\R$, we define the modular $|f|_{p(\cdot)}$ of $f$ by
\[|f|_{p(\cdot)}=\int_{\Omega}|f(x)|^{p(x)}\dx.\] 
Then we define the following variable-exponent Lebesgue spaces 
\begin{equation*}
L^{p(\cdot)}(\Omega)\defeq \left\{u\in L^1_{\rm{loc}}(\Omega):|u|_{p(\cdot)}<\infty\right\},
\end{equation*}
with the corresponding Luxembourg norms
\[
\norm{u}_{L^{p(\cdot)}(\Omega)}\defeq \inf\left\{\lambda>0:\bigg|\frac{u(x)}{\lambda}\bigg|_{p(\cdot)}\leq1\right\}
\]
If we assume that $1<p^-\leq p^+<\infty$, it is well-known that  the variable-exponent space $L^{p(\cdot)}(\Omega)$ is a reflexive, separable Banach space. We assume that  $1<p^-\leq p^+<\infty$ throughout the present paper.

In order to exploit various results from the theory of classical Lebesgue space, some regularity of the exponent $p(\cdot)$ is needed:  {\textit{{\rm{log}}-H\"older continuity}}.
\begin{definition}
We say a function $p:\Omega\rightarrow\R$ is {\textit{locally {\rm{log}}-H\"older continuous}} on $\Omega$ if there exists a positive constant $C_1>0$ satisfying for all $x,y\in\Omega$,
\begin{equation}\label{local_log_holder}
|p(x)-p(y)|\leq\frac{C_1}{{\rm{log}}(e+1/|x-y|)}.
\end{equation}
Furthermore, we call that $p$ satisfies the {\textit{{\rm{log}}-H\"older decay condition}} if there exist some constants $p_{\infty}\in\R$ and $C_2>0$ such that for all $x\in\Omega,$
\begin{equation}\label{log_holder_decay}
|p(x)-p_{\infty}|\leq\frac{C_2}{{\rm{log}}(e+|x|)}.
\end{equation}
We call that $p$ is {\textit{globally {\rm{log}}-H\"older continuous}} in $\Omega$ if both \eqref{local_log_holder} and \eqref{log_holder_decay} hold. We say $C_{\rm{log}}(p)\defeq\max\{C_1,C_2\}$ a log-{\textit{H\"older constant}} of $p$.
\end{definition}

\begin{definition}
We define the family of log-H\"older continuous variable-exponent functions:
\[\mathcal{P}^{\rm{log}}(\Omega)\defeq\left\{p\in\mathcal{P}(\Omega):\frac{1}{p}\,\,{\rm{is}}\,\,{\rm{globally}}\,\,{\textrm{log-H\"older}}\,\,{\rm{continuous}}\right\}.\]
If $\Omega$ is unbounded, we define $p_{\infty}$ by $\frac{1}{p_{\infty}}\defeq\lim_{|x|\rightarrow\infty}\frac{1}{p(x)}$.
\end{definition}
Here we note that, since $p\mapsto\frac{1}{p}$ is a bilipschitz mapping from $[p^-,p^+]$ to $\left[\frac{1}{p^+},\frac{1}{p^-}\right]$, if $p\in\mathcal{P}(\Omega)$ with $p^+<\infty$, it follows that $p\in\mathcal{P}^{\rm{log}}(\Omega)$ if and only if $p$ is globally H\"older continuous. For more details, one may see \cite{DHHR2011} as a comprehensive source of information for the theory of variable-exponent spaces.

Now, we introduce the following notation
\[\overline{\DD}\uu\defeq(1+|\DD\uu|^2)^{\frac{1}{2}},\]
and define the some special energies which will be used in the paper:
\begin{align*}
\mathcal{I}_{p}(\uu)(t) &=\int_{\R^d}(\overline{\DD}\uu)^{p(\cdot)-2}|\DD\uu|^2\dx, \\
\mathcal{J}_{p}(\uu)(t) &=\int_{\R^d}(\overline{\DD}\uu)^{p(\cdot)-2}|\nabla\DD\uu|^2\dx. \\
\end{align*}
Note that, since $|\nabla^2\uu|\leq3|\nabla\DD\uu|\leq6|\nabla^2\uu|$, $|\nabla\DD\uu|$ can be replaced by $|\nabla^2\uu|$ (and vise versa) in any appearance with the multiplication of some suitable constants.

In the present paper, for the simple presentation, we only focus on the case of three space dimensions. Note, however, that the theorem for the general case follows with same arguments. The existence of strong solutions of the equations \eqref{eq1}-\eqref{initial} can be found in various literatures. For instance, in \cite{strong_sol_per}, the local existence of strong solutions was proved under the condition $\frac{7}{5}<p^-\leq p^+\leq 2$ in a bounded domain with periodic boundary conditions. For the case of the whole space $\R^3$, in \cite{Cauchy} the model with a constant exponent $p$ was considered, and the existence of global-in-time strong solutions was established assuming that $p\geq\frac{11}{5}$.  For the case of variable-exponent power-law index, we can prove that the strong solutions exist in $\R^3$ to the equations \eqref{eq1}, \eqref{eq2} with
\begin{equation}\label{xdep}
\SSS(\DD\uu)=(1+|\DD\uu|^2)^{\frac{p(x)-2}{2}}\DD\uu,
\end{equation}
where the power-law index $p(\cdot)$ depends only on the spatial variable $x\in\R^3$. In fact, by assuming that $p\in W^{1,\infty}(\R^3)\cap\mathcal{P}^{\rm{log}}(\R^3)$ with $p^-\geq\frac{11}{5}$ and $\uu_0\in H^1(\R^3)$, if we follow the proof and apply the same arguments presented in \cite{Cauchy}, the result in \cite{Cauchy} can be extended to the case of variable power-law index in a straightforward manner and we obtain the existence of the global strong solutions of \eqref{eq1}, \eqref{eq2} and \eqref{xdep}. For further consideration, based on the arguments presented in \cite{Malek}, we can also prove that the global strong solutions to \eqref{eq1}, \eqref{eq2} and \eqref{xdep} exist provided that $p\in W^{1,\infty}(\R^3)\cap\mathcal{P}^{\rm{log}}(\R^3)$ with $p^->\frac{5}{3}$, under the smallness assumption of the initial data. The assumption that $p(\cdot)$ only depends of the space variable $x\in\R^3$ is important, as in this case, we can take time derivative  in the same way as for the constant-exponent case, and the existence of strong solutions follows in  a similar manner. The existence of global strong solutions of the model \eqref{eq1}-\eqref{ST} in the whole domain $\R^3$ where $p=p(t,x)$ depends on both time and space variables is still an open problem.  Here for the strong solutions of the equations \eqref{eq1}, \eqref{eq2} and \eqref{xdep}, we mean that
\[\uu\in L^{\infty}((0,T);H^1(\R^3)^3)\cap L^2((0,T); H^2(\R^3)^3),\]
\[|\nabla\uu|\in L^{\infty}((0,T);L^{p(\cdot)}(\R^3))\quad{\rm{and}}\quad \partial_t\uu\in L^2(Q_T)^3,\]
with the following energy inequalities:
\begin{align}
\sup_{0\leq t\leq T}\|\uu(t)\|^2_2+\int_0^T\mathcal{I}_{p}(\uu)(t)\dt &\leq \|\uu_0\|^2_2\label{energy_1},\\
\sup_{0\leq t\leq T}\|\nabla\uu(t)\|^2_2+\int_0^T\mathcal{J}_{p}(\uu)(t)\dt &\leq C(\|\uu_0\|_{H^1})\label{energy_2}.
\end{align}

In the present paper, we consider the aforementioned global-in-time strong solutions of the equations \eqref{eq1}, \eqref{eq2} and \eqref{xdep} under the the condition $p\in W^{1,\infty}(\R^3)\cap\mathcal{P}^{\rm{log}}(\R^3)$ and $p^-\geq\frac{11}{5}$. Note however, that the proof of our main theorem may also apply to the models with space-time-dependent power-law index, and therefore we can extend our main result to the case of the equations \eqref{eq1}-\eqref{ST} once the existence of corresponding strong solutions for $p=p(t,x)$ is established.

Before presenting our main result,  we recall the following decay estimate of the strong solutions of \eqref{eq1}, \eqref{eq2} and \eqref{xdep} , which is quoted from \cite{Ko2022}.
\begin{theorem}\label{prev_thm}
Assume that $\uu_0\in L^1(\R^3)\cap H^1(\R^3)$, and suppose that $p^-\geq\frac{11}{5}$. Then for the strong solutions $\uu$ of \eqref{eq1}, \eqref{eq2} and \eqref{xdep}, there exists some constant $C>0$ depending on the $L^1$ and $H^1$-norms of $\uu_0$ such that the following decay estimate holds:
\[\|\uu(t)\|_2\leq C(1+t)^{-\frac{3}{4}}\quad\forall t>0.\]
\end{theorem}

Now we are ready to state our main theorems. Note here that the assumption $p\in\mathcal{P}^{\rm{log}}(\R^3)\cap W^{1,\infty}(\R^3)$ is needed for the existence of strong solutions, and not for the proof of Theorem \ref{main_thm}, and therefore will be omitted in the statement of our main theorem.
\begin{theorem}\label{main_thm}
Suppose that $\uu_0\in L^1(\R^3)\cap H^1(\R^3)$ and the initial perturbation $\ww_0\in L^1(\R^3)\cap H^1(\R^3)$ satisfies the decay condition \eqref{initial_decay} with  $2<\gamma<\frac{5}{2}$. Assume further that $p^-\geq\frac{17}{7}$. If we let $\uu$ be the strong solutions of the original problem \eqref{eq1}, \eqref{eq2} and \eqref{xdep}, and $\uuu$ be the strong solution of the perturbed problem \eqref{p_eq1}-\eqref{p_initial} with \eqref{xdep}, then we have the following optimal upper and lower bounds: for sufficiently large $t>0$,
\[C_1(1+t)^{-\frac{\gamma}{2}}\leq\|\uuu(t)-\uu(t)\|_2\leq C_2(1+t)^{-\frac{\gamma}{2}}.\]
\end{theorem}
\end{section}

\begin{section}{Estimate for upper bounds}
We first define $\GGG(\DD\uu)\defeq \left((\DDD\uu)^{p(x)-2}-1\right)\DD\uu$ and rewrite the equations \eqref{p_eq1} and \eqref{eq1} as
\begin{align}
\partial_t\uuu+(\uuu\cdot\nabla)\uuu-{\rm{div}}\,\GGG(\DD\uuu)-\Delta\uuu&=-\nabla \tilde{\pi}+\boldsymbol{f}, \label{re_eq1} \\
\partial_t\uu+(\uu\cdot\nabla)\uu-{\rm{div}}\,\GGG(\DD\uu)-\Delta\uu&=-\nabla \pi+\boldsymbol{f}. \label{re_eq2}
\end{align}
Let us denote the difference between the perturbed strong solutions of the equations \eqref{p_eq1}-\eqref{p_initial} and the strong solutions of \eqref{eq1}-\eqref{initial} by
\begin{equation}\label{sol_diff}
	\ww(t,x)=\uuu(t,x)-\uu(t,x)\quad{\rm{and}}\quad r(t,x)=\tilde{\pi}(t,x)-\pi(t,x).
\end{equation}
Then we observe that $\ww(t,x)$ and $r(t,x)$ satisfy the following equations in the sense of distribution:
\begin{align}
\partial_t\ww+(\uuu\cdot\nabla)\ww+(\ww\cdot\nabla)\uu-\Delta\ww+\nabla r&={\rm{div}}\,\GGG(\DD\uuu)-{\rm{div}}\,\GGG(\DD\uu)\label{d_eq1}\\
{\rm{div}}\,\ww&=0, \label{d_eq2}\\
\ww(0,x)&=\ww_0(x).\label{d_initial}
\end{align}

We begin with the following lemma regarding the monotonicity of $\GGG$. In fact, various properties of the extra stress tensor $\SSS$ can be described in the framework of $p(\cdot)$-potential define by
\[F(x,s)=\int^s_0(1+a^2)^{\frac{p(x)-2}{2}}a\,{\rm{d}}s\quad{\rm{and}}\quad\Phi(x,\boldsymbol{B})=F(x,|\boldsymbol{B}|).\]
Here however, instead of presenting general theory based on $p(\cdot)$-potential, we only discuss the monotonicity of $\SSS$ and $\GGG$, which are required in our proof. For completeness, we also include the proof for the monotonicity of $\GGG$.

\begin{lemma}\label{G_mon}
Let $\GGG:\R^{d\times d}_{\rm{sym}}\rightarrow \R^{d\times d}_{\rm{sym}}$ be defined by $\GGG(\DD\uu)\defeq \left((\DDD\uu)^{p(x)-2}-1\right)\DD\uu$, where $\SSS$ is the extra stress tensor defined by \eqref{xdep}. If $p^-\geq 2$, then $\GGG$ is monotone in the sense that for all $\BB,\CC\in\R^{d\times d}_{\rm{sym}}$, we have
\[\left(\GGG(\BB)-\GGG(\CC)\right):(\BB-\CC)\geq 0.\]
\end{lemma}
\begin{proof}
For $\BB,\CC\in\R^{d\times d}_{\rm{sym}}$, let us denote $\MM(s)=\CC+s(\BB-\CC)$. Then we have
\begin{align*}
&(\SSS(\BB)-\SSS(\CC)):(\BB-\CC) \\
& =\sum_{i,j}\left(\SSS(\BB)-\SSS(\CC)\right)_{ij}(\BB-\CC)_{ij}\\
& =\sum_{i,j,k,\ell}\int^1_0\partial_{k\ell}\left(\MM(s)_{ij}\right)\ds(\BB-\CC)_{ij}(\BB-\CC)_{k\ell}\\
& = \sum_{i,j,k,\ell}\int^1_0\left((p(x)-2)(1+|\MM|^2)^{\frac{p(x)-4}{2}}\MM_{k\ell}\MM_{ij}+(1+|\MM|^2)^{\frac{p(x)-2}{2}}\partial_{k\ell}\MM_{ij}\right)\ds (\BB-\CC)_{ij}(\BB-\CC)_{k\ell}\\
& =\int^1_0\left((p(x)-2)(1+|\MM|^2)^{\frac{p(x)-4}{2}}|\MM:(\BB-\CC)|^2+(1+|\MM|^2)^{\frac{p(x)-2}{2}}|\BB-\CC|^2\right)\ds\\
& \geq\int^1_0(1+|\MM|^2)^{\frac{p(x)-2}{2}}|\BB-\CC|^2\ds\geq|\BB-\CC|^2.
\end{align*}
Therefore, we conclude that for all $\BB,\CC\in\R^{d\times d}_{\rm{sym}}$, we have
\[\left(\GGG(\BB)-\GGG(\CC)\right):(\BB-\CC)=\left(\SSS(\BB)-\SSS(\CC)\right):(\BB-\CC)-|\BB-\CC|^2\geq0,\]
which implies the monotonicity of $\GGG$.
\end{proof}

Next, we shall derive an $L^1$ estimate for the modified stress tensor $\GGG$. Though the condition required in our main theorem is $p^-\geq\frac{17}{7}$, here however, we deal with a general discussion with $p^-\geq\frac{11}{5}$ where the existence of solutions if guaranteed. We will mention in  later analysis when the restricted condition $p^-\geq\frac{17}{7}$ is actually needed.
\begin{lemma}\label{main_est}
Assume that $p^-\geq\frac{11}{5}$ and $\vv$ is sufficiently smooth. Then there exist positive constants $C_1$, $C_2$ and $C_3$ such that the following inequality holds for all $t\in(0,T)$:
\begin{itemize}
\item {\rm{(Case\,\,$1$)}} $p^-\geq3$:
\[\int^t_0\int_{\R^3}|\GGG(\DD\vv)|\dx\ds\leq C_1,\]
\item {\rm{(Case\,\,$2$)}} $\frac{11}{5}\leq p^-<3$:
\[\int^t_0\int_{\R^3}|\GGG(\DD\vv)|\dx\ds\leq C_2 + C_3 \left(\int^t_0\|\vv(s)\|_2^{\frac{2\alpha}{2-\beta}}\ds\right)^{\frac{2-\beta}{2}},\]
\end{itemize}
where $\alpha=\frac{7-p^-}{4}$ and $\beta=\frac{5p^--11}{4}$.
\end{lemma}

\begin{proof}
By observing the inequality $(1+s)^{\alpha}-1\lesssim s+s^{\alpha}$ for $s\geq0$ and $\alpha>0$, we have that
\begin{align*}
\int_{\R^3}|\GGG(\DD\vv)|\dx
&\lesssim\int_{\R^3}\left((1+|\DD\vv|)^{p(x)-2}-1\right)|\DD\vv|\dx\\
&\lesssim\int_{\R^3}\left(|\DD\uu|+|\DD\uu|^{p(x)-2}\right)|\DD\uu|\dx\\
&\lesssim \int_{\R^3}|\DD\vv|^2\dx+\int_{\R^3}|\DD\vv|^{p(x)-1}\dx\\
&\lesssim \mathcal{I}_p(\vv)+\int_{\{|\DD\vv|\geq1\}}|\DD\vv|^{p(x)-1}\dx + \int_{\{|\DD\vv|<1\}}|\DD\vv|^{p(x)-1}\dx\\
&\lesssim \mathcal{I}_p(\vv)+\int_{\R^3}|\DD\vv|^{p(x)}\dx+\int_{\R^3}|\DD\vv|^{p^--1}\dx\\
&\lesssim \mathcal{I}_p(\vv)+\int_{\R^3}|\DD\vv|^{p^--1}\dx.
\end{align*}
Now it is enough to estimate the second integral on the right-hand side. For the case $p^-\geq 3$, by \eqref{energy_1} and the interpolation inequality,
\begin{align*}
\int^t_0\|\nabla\vv(s)\|^{p^--1}_{p^--1}\ds &\lesssim\int^t_0\|\nabla\vv(s)\|_2^{\frac{2}{p^--2}}  \|\nabla\vv(s)\|^{\frac{p^-(p^--3)}{p^--2}}_{p^-}\ds  \\
& \lesssim \|\nabla\vv\|_{L^2((0,T);L^2)}^{\frac{2}{p^--2}}  \|\nabla\vv\|_{L^{p^-}((0,T);L^{p^-})}^{\frac{p^-(p^--3)}{p^--2}} <\infty.
\end{align*}

For the case of $\frac{11}{5}\leq p^-<3$, by \eqref{energy_2} and Gagliardo--Nirenberg interpolation inequality (see, for example, \cite{GNI})
\begin{align*}
\int^t_0\|\nabla\vv(s)\|^{p^--1}_{p^--1}\ds 
& \lesssim  \int^t_0\|\vv(s)\|_2^{\alpha}\|\nabla^2\vv(s)\|_2^{\beta}\ds\\
& \lesssim \left(\int^t_0\|\vv(s)\|_2^{\frac{2\alpha}{2-\beta}}\ds\right)^{\frac{2-\beta}{2}}\left(\int^t_0\|\nabla^2\vv(s)\|_2^2\ds\right)^{\frac{\beta}{2}}\\
& \lesssim \left(\int^t_0\|\vv(s)\|_2^{\frac{2\alpha}{2-\beta}}\ds\right)^{\frac{2-\beta}{2}},
\end{align*}
where $\alpha=\frac{7-p^-}{4}$ and $\beta=\frac{5p^--11}{4}$, which completes the proof.
\end{proof}

Next, we shall derive the following estimate, which is needed for the Fourier splitting method.
\begin{lemma}\label{second_est}
Assume that $\uu_0\in H^1(\R^3)\cap L^1(\R^3)$ and $p^-\geq\frac{11}{5}$. Then for $\ww(t,x)$ defined in \eqref{sol_diff}, we have the following estimates:

\begin{itemize}
\item {\rm{(Case\,\,$1$)}} $p^-\geq3$:
\begin{equation}\label{repeat_1}
|\widehat{\ww}(t,\xi)|\lesssim |\widehat{\ww}_0(\xi)|+|\xi|\left(1+\int^t_0\|\ww(s)\|_2\ds\right).
\end{equation}
\item {\rm{(Case\,\,$2$)}} $\frac{11}{5}\leq p^-<3$:
\begin{equation}\label{repeat_2}
|\widehat{\ww}(t,\xi)|\lesssim |\widehat{\ww}_0(\xi)|+|\xi|\left(1+\left(\int^t_0\|\uuu(s)\|_2^{\frac{2\alpha}{2-\beta}}\ds\right)^{\frac{2-\beta}{2}} + \left(\int^t_0\|\uu(s)\|_2^{\frac{2\alpha}{2-\beta}}\ds\right)^{\frac{2-\beta}{2}}  +\int^t_0\|\ww(s)\|_2\ds\right).
\end{equation}
\end{itemize}
\end{lemma}

\begin{proof}
Taking the Fourier transformation on \eqref{d_eq1} yields
\begin{equation}\label{ode}
\widehat{\ww}_t+|\xi|^2\widehat{\ww}=F(t,\xi)\quad{\rm{and}}\quad\widehat{\ww}_0(\xi)\defeq\widehat{\ww}(0,\xi)=\widehat{\ww}_0,
\end{equation}
where
\begin{equation}\label{aid_1}
F(t,\xi)\defeq\widehat{\nabla\cdot\GGG(\DD\uuu)}(t,\xi)-\widehat{\nabla\cdot\GGG(\DD\uu)}(t,\xi)-\widehat{(\uuu\cdot\nabla)\ww}(t,\xi)-\widehat{(\ww\cdot\nabla)\uu}(t,\xi)-\widehat{\nabla r}(t,\xi).
\end{equation}
For the initial condition, note that
\begin{equation}\label{ini_bd}
|\widehat{\ww}_0(\xi)|\leq \bigg|\int_{\R^3}e^{-ix\cdot\xi}\ww_0(x)\dx\bigg|\leq\int_{\R^3}|\ww_0(x)|\dx\leq C,
\end{equation}
Regarding the stress tensor terms, we observe that
\begin{equation}\label{aid_2_1}
\big|\widehat{\nabla\cdot\GGG(\DD\uuu)}(t,\xi)\big| = \bigg|\int_{\R^3}e^{-ix\cdot\xi}\nabla\cdot\GGG(\DD\uuu)\dx\bigg| \leq|\xi|\int_{\R^3}|\GGG(\DD\uuu)|\dx,
\end{equation}
and
\begin{equation}\label{aid_2_2}
\big|\widehat{\nabla\cdot\GGG(\DD\uu)}(t,\xi)\big| = \bigg|\int_{\R^3}e^{-ix\cdot\xi}\nabla\cdot\GGG(\DD\uu)\dx\bigg| \leq|\xi|\int_{\R^3}|\GGG(\DD\uu)|\dx.
\end{equation}

Furthermore, by using H\"older's inequality together with \eqref{eq2}, \eqref{p_eq2} and \eqref{d_eq2} we obtain
\begin{equation}\label{aid_3_1}
\big|\widehat{(\uuu\cdot\nabla)\ww}(t,\xi)\big|=\bigg|\int_{\R^3}e^{-ix\cdot\xi}\nabla\cdot(\uuu\otimes\ww)\dx\bigg|\leq|\xi|\|\uuu(t)\otimes\ww(t)\|_1\leq|\xi|\|\uuu(t)\|_2\|\ww(t)\|_2,
\end{equation}
and similarly we have
\begin{equation}\label{aid_3_2}
\big|\widehat{(\ww\cdot\nabla)\uu}(t,\xi)\big|=\bigg|\int_{\R^3}e^{-ix\cdot\xi}\nabla\cdot(\ww\otimes\uu)\dx\bigg|\leq|\xi|\|\ww(t)\otimes\uu(t)\|_1\leq|\xi|\|\ww(t)\|_2\|\uu(t)\|_2.
\end{equation}

Next, by taking divergence operator on \eqref{d_eq1}, we deduce that
\[\Delta r=\sum_{i,j}\frac{\partial^2}{\partial x_i \partial x_j}(-\uuu_i\ww_j-\ww_i\uu_j+\GGG(\DD\uuu)_{ij}-\GGG(\DD\uuu)_{ij}),\]
and therefore, by H\"older's inequality, we have
\begin{equation}\label{aid_4}
|\widehat{\nabla r}(t,\xi)|\leq    |\xi|\|\ww(t)\|_2\left(\|\uuu(t)\|_2+\|\uu(t)\|_2\right)+|\xi|\left(\|\GGG(\DD\uuu)\|_1+\|\GGG(\DD\uu)\|_1\right).
\end{equation}
Now, it follows from \eqref{ode} that
\[\widehat{\ww}(t,\xi)=e^{-|\xi|^2t}\widehat{\ww}_0(\xi)+\int^t_0F(s,\xi)e^{-|\xi|^2(t-s)}\ds.\]
Therefore, by \eqref{aid_1}-\eqref{aid_4}, Lemma \ref{main_est} and the fact that $\sup_{t>0}\left(\|\uuu(t)\|_2+\|\uu(t)\|_2\right) \leq C$,  we obtain the desired result.
\end{proof}

Now we prove the estimate for the upper bound in Theorem \ref{main_thm}. Multiplying \eqref{d_eq1} by $\ww$ and integrating over time yields that
\begin{align*}
\frac{1}{2}\frac{\rm{d}}{\dt}\|\ww\|^2_2+\|\nabla\ww\|^2_2
& = \int_{\R^3}\GGG(\DD\uu):\DD\ww\dx-\int_{\R^3}\GGG(\DD\uuu):\DD\ww\dx-\int_{\R^3}(\ww\cdot\nabla)\uu\cdot\ww\dx \\
& = -\int_{\R^3}\left(\GGG(\DD\uuu)-\GGG(\DD\uu)\right):(\DD\uuu-\DD\uu)\dx-\int_{\R^3}(\ww\cdot\nabla)\uu\cdot\ww\dx.
\end{align*}
The first term on the right-hand side is negative by Lemma \ref{G_mon}. For the second term on the right-hand side, by the interpolation inequality, Sobolev embedding and Young's inequality, we have
\begin{align*}
-\int_{\R^3}(\ww\cdot\nabla)\uu\cdot\ww\dx
& \leq \|\ww\|_2\|\nabla\uu\|_3\|\ww\|_6 \\
& \leq \|\ww\|_2\|\nabla\uu\|^{\frac{1}{2}}_2\|\nabla\uu\|^{\frac{1}{2}}_6\|\nabla\ww\|_2\\
& \leq  C(\varepsilon)\|\ww\|^2_2\|\nabla\uu\|_2\|\nabla^2\uu\|_2+\varepsilon\|\nabla\ww\|^2_2\\
& \leq C(\varepsilon)\|\ww\|^2_2\|\nabla\uu\|^2_2+C(\varepsilon)\|\ww\|^2_2\|\nabla^2\uu\|^2_2+\varepsilon\|\nabla\ww\|^2_2.
\end{align*}
Therefore, we obtain the inequality
\begin{equation}\label{d_energy}
\frac{\rm{d}}{\dt}\|\ww\|^2_2+\|\nabla\ww\|^2_2\lesssim\left(\|\nabla\uu\|^2_2+\|\nabla^2\uu\|^2_2\right)\|\ww\|^2_2.
\end{equation}

Now, let us denote $g(\uu,\ww)=\left(\|\nabla\uu\|^2_2+\|\nabla^2\uu\|^2_2\right)\|\ww\|^2_2$ and assume that $f(t)$ is a smooth function satisfying $f(0)=1$, $f(t)>0$ and $f'(t)>0$. If we multiply both sides of \eqref{d_energy} by $f(t)$, by Plancherel's theorem, for some constant $C_0>0$, we have
\begin{equation}\label{mid_ineq_const}
\frac{\rm{d}}{\dt}\left(f(t)\int_{\R^3}|\widehat{\ww}(t,\xi)|^2\dxi\right)+C_0f(t)\int_{\R^3}|\xi|^2|\widehat{\ww}(t,\xi)|^2\dxi\leq f'(t)\int_{\R^3}|\widehat{\ww}(t,\xi)|^2\dxi+f(t)g(\uu,\ww).
\end{equation}
We shall define the set $L(t)\defeq\{\xi\in\R^3:C_0|\xi|^2f(t)\leq f'(t)\}$ where $C_0>0$ is the constant in \eqref{mid_ineq_const}. Then we have
\[C_0f(t)\int_{\R^3}|\xi|^2|\widehat{\ww}(t,\xi)|^2{\rm{d}}\xi\geq C_0f(t)\int_{L(t)^{c}}|\xi|^2|\widehat{\ww}(t,\xi)|^2{\rm{d}}\xi\geq f'(t)\int_{\R^3}|\widehat{\ww}(t,\xi)|^2{\rm{d}}\xi-f'(t)\int_{L(t)}|\widehat{\ww}(t,\xi)|^2{\rm{d}}\xi,\]
Therefore, we have
\begin{equation}\label{mid_ineq_const_2}
\frac{\rm{d}}{\dt}\left(f(t)\int_{\R^3}|\widehat{\ww}(t,\xi)|^2\dxi\right)\leq f'(t)\int_{L(t)}|\widehat{\ww}(t,\xi)|^2\dxi+f(t)g(\uu,\ww).
\end{equation}
If we integrate \eqref{mid_ineq_const_2} over $(0,t)$, it follows that
\begin{equation}\label{mid_est}
f(t)\int_{\R^3}|\widehat{\ww}(t,\xi)|^2{\rm{d}}\xi\leq\int_{\R^3}|\widehat{\ww}_0(\xi)|^2{\rm{d}}\xi+\int^t_0 f'(s)\int_{L(s)}|\widehat{\ww}(s,\xi)|^2{\rm{d}}\xi\ds+\int^t_0f(s)g(\uu,\ww)\ds.
\end{equation}
(Case $1$) $p^-\geq3$: Now we set $f(t)=(1+t)^{4+\gamma}$. Then by Lemma \ref{second_est} with \eqref{mid_est} and Lemma \ref{linear_decay}, we obtain that
\begin{align*}
(1+t)^{4+\gamma}\|\ww(t)\|^2_2
& \lesssim\int_{\R^3}|\widehat{\ww}_0(\xi)|^2{\rm{d}}\xi+\int^t_0(1+s)^{3+\gamma}\int_{L(s)}|\widehat{e^{\Delta s}\ww_0}|^2{\rm{d}}\xi\ds\\
& \hspace{5mm}+\int^t_0(1+s)^{3+\gamma}\int_{L(s)}|\xi|^2\left(1+\int^s_0\|\ww(\tau)\|_2{\rm{d}}\tau\right)^2{\rm{d}}\xi\ds+\int^t_0(1+s)^{4+\gamma}g(\uu,\ww)\ds\\
& \lesssim1+\int^t_0(1+s)^{3+\gamma}\|e^{\Delta s}\ww_0\|^2_2\ds+\int^t_0(1+s)^{3+\gamma}\int_{L(s)}|\xi|^2\dxi\ds\\
& \hspace{5mm}+\int^t_0(1+s)^{3+\gamma}\int_{L(s)}|\xi|^2s^2{\rm{d}}\xi\ds+\int^t_0(1+s)^{4+\gamma}g(\uu,\ww)\ds\\
& \lesssim 1+\int^t_0(1+s)^{3}\ds+\int^t_0(1+s)^{\frac{1}{2}+\gamma}\ds +\int^t_0(1+s)^{\frac{5}{2}+\gamma}\ds+\int^t_0(1+s)^{4+\gamma}g(\uu,\ww)\ds\\
& \lesssim 1+(1+t)^{4}+(1+t)^{\frac{3}{2}+\gamma}+(1+t)^{\frac{7}{2}+\gamma}+\int^t_0(1+s)^{4+\gamma}g(\uu,\ww)\ds\\
& \lesssim (1+t)^{\frac{7}{2}+\gamma}+\int^t_0(1+s)^{4+\gamma}(\|\nabla\uu(s)\|^2_2+\|\nabla^2\uu(s)\|^2_2)\|\ww(s)\|^2_2\ds.
\end{align*}
Then by Gronwall's inequality, we obtain
\begin{align*}
(1+t)^{4+\gamma}\|\ww(t)\|^2_2
& \lesssim (1+t)^{\frac{7}{2}+\gamma}\exp\left(\int^{\infty}_0(\|\nabla\uu(s)\|_2^2+\|\nabla^2\uu(s)\|^2_2)\ds\right)\\
& \lesssim (1+t)^{\frac{7}{2}+\gamma}\exp\left(\int^{\infty}_0(\mathcal{I}_p(\uu)+\mathcal{J}_p(\uu))\ds\right),
\end{align*}
which, by \eqref{energy_1} and \eqref{energy_2}, implies that the decay rate
\begin{equation}\label{iter_1}
\|\ww(t)\|_2\lesssim(1+t)^{-\frac{1}{4}},\quad\forall t>0.
\end{equation}

Next, we substitute \eqref{iter_1} into \eqref{repeat_1} and proceed as above.  Then we have
\begin{align*}
(1+t)^{4+\gamma}\|\ww(t)\|^2_2
& \lesssim\int_{\R^3}|\widehat{\ww}_0(\xi)|^2{\rm{d}}\xi+\int^t_0(1+s)^{3+\gamma}\int_{L(s)}|\widehat{e^{\Delta s}\ww_0}|^2{\rm{d}}\xi\ds\\
& \hspace{5mm}+\int^t_0(1+s)^{3+\gamma}\int_{L(s)}|\xi|^2\left(1+\int^s_0(1+\tau)^{-\frac{1}{4}}{\rm{d}}\tau\right)^2{\rm{d}}\xi\ds+\int^t_0(1+s)^{4+\gamma}g(\uu,\ww)\ds\\
& \lesssim 1+\int^t_0(1+s)^{3}\ds+\int^t_0(1+s)^{\frac{1}{2}+\gamma}\ds +\int^t_0(1+s)^{2+\gamma}\ds+\int^t_0(1+s)^{4+\gamma}g(\uu,\ww)\ds\\
& \lesssim 1+(1+t)^{4}+(1+t)^{\frac{3}{2}+\gamma}+(1+t)^{3+\gamma}+\int^t_0(1+s)^{4+\gamma}g(\uu,\ww)\ds\\
& \lesssim (1+t)^{3+\gamma}+\int^t_0(1+s)^{4+\gamma}(\|\nabla\uu(s)\|^2_2+\|\nabla^2\uu(s)\|^2_2)\|\ww(s)\|^2_2\ds,
\end{align*}
from which we deduce that by Gronwall's inequality
\begin{equation}\label{iter_2}
\|\ww(t)\|_2\lesssim(1+t)^{-\frac{1}{2}},\quad\forall t>0.
\end{equation}

We shall repeat this process few more times. In the same way, putting \eqref{iter_2} into \eqref{repeat_1} yields
\begin{align*}
(1+t)^{4+\gamma}\|\ww(t)\|^2_2
& \lesssim (1+t)^{\frac{5}{2}+\gamma}+\int^t_0(1+s)^{4+\gamma}(\|\nabla\uu(s)\|^2_2+\|\nabla^2\uu(s)\|^2_2)\|\ww(s)\|^2_2\ds,
\end{align*}
which leads us to
\begin{equation}\label{iter_3}
\|\ww(t)\|_2\lesssim(1+t)^{-\frac{3}{4}},\quad\forall t>0.
\end{equation}
Likewise, we have in the same way that
\begin{align*}
(1+t)^{4+\gamma}\|\ww(t)\|^2_2
& \lesssim (1+t)^{2+\gamma}+\int^t_0(1+s)^{4+\gamma}(\|\nabla\uu(s)\|^2_2+\|\nabla^2\uu(s)\|^2_2)\|\ww(s)\|^2_2\ds,
\end{align*}
and
\begin{equation}\label{iter_3}
\|\ww(t)\|_2\lesssim(1+t)^{-1},\quad\forall t>0.
\end{equation}

Finally, if we repeat this process once more, we obtain that
\begin{align*}
(1+t)^{4+\gamma}\|\ww(t)\|^2_2
& \lesssim 1+\int^t_0(1+s)^{3}\ds+\int^t_0(1+s)^{\frac{1}{2}+\gamma}\ds\\
& \hspace{5mm}+\int^t_0(1+s)^{\frac{1}{2}+\gamma}\left(\log(1+s)\right)^2\ds+\int^t_0(1+s)^{4+\gamma}g(\uu,\ww)\ds\\
& \lesssim 1+(1+t)^{4}+(1+t)^{\frac{3}{2}+\gamma}+(1+t)^{4}+\int^t_0(1+s)^{4+\gamma}g(\uu,\ww)\ds\\
& \lesssim (1+t)^{4}+\int^t_0(1+s)^{4+\gamma}(\|\nabla\uu(s)\|^2_2+\|\nabla^2\uu(s)\|^2_2)\|\ww(s)\|^2_2\ds,
\end{align*}
where we have used the fact that $\gamma<\frac{5}{2}$ and $\log(1+x)\leq C_{\alpha} x^{\alpha}$ for arbitrarily small $\alpha>0$.
Therefore, we finally obtain that the desired decay rate
\begin{equation}\label{iter_final}
\|\uuu(t)-\uu(t)\|_2=\|\ww(t)\|_2\lesssim(1+t)^{-\frac{\gamma}{2}},\quad\forall t>0.
\end{equation}
(Case $2$) $\frac{17}{7}\leq p^-<3$: In this case, we first note that $\frac{1}{2}<\frac{2-\beta}{2}<1$ and $\frac{4\alpha}{2-\beta}>2$. Then by H\"older's inequality and \eqref{energy_1}, we obtain
\begin{align*}
\int^t_0(  &  1+s)^{3+\gamma}\int_{L(s)}|\xi|^2\left(1+\left(\int^s_0\|\uuu(\tau)\|_2^{\frac{2\alpha}{2-\beta}}{\rm{d}}\tau\right)^{\frac{2-\beta}{2}}+ \left(\int^s_0\|\uu(\tau)\|_2^{\frac{2\alpha}{2-\beta}}{\rm{d}}\tau\right)^{\frac{2-\beta}{2}}+\int^s_0\|\ww(\tau)\|_2{\rm{d}}\tau\right)^2{\rm{d}}\xi\ds\\
&\lesssim \int^t_0(1+s)^{3+\gamma}\int_{L(s)}|\xi|^2\,{\rm{d}}\xi\ds + \int^t_0(1+s)^{3+\gamma}\int_{L(s)}|\xi|^2\left(\int^s_0\|\uuu(\tau)\|_2^{\frac{2\alpha}{2-\beta}}\,{\rm{d}}\tau\right)^{2-\beta}\,{\rm{d}}\xi\ds\\
&\hspace{5mm}+ \int^t_0(1+s)^{3+\gamma}\int_{L(s)}|\xi|^2\left(\int^s_0\|\uu(\tau)\|_2^{\frac{2\alpha}{2-\beta}}\,{\rm{d}}\tau\right)^{2-\beta}\,{\rm{d}}\xi\ds\\
&\hspace{5mm}+\int^t_0(1+s)^{3+\gamma}\int_{L(s)}|\xi|^2\left(\int^s_0\|\ww(\tau)\|_2\,{\rm{d}}\tau\right)^2\,{\rm{d}}\xi\ds\\
&\lesssim \int^t_0(1+s)^{3+\gamma}\int_{L(s)}|\xi|^2\,{\rm{d}}\xi\ds + \int^t_0(1+s)^{3+\gamma}\int_{L(s)}|\xi|^2s^{\frac{2-\beta}{2}}\left(\int^s_0\|\uuu(\tau)\|_2^{\frac{4\alpha}{2-\beta}}\,{\rm{d}}\tau\right)^{\frac{2-\beta}{2}}\,{\rm{d}}\xi\ds\\
&\hspace{5mm}+ \int^t_0(1+s)^{3+\gamma}\int_{L(s)}|\xi|^2s^{\frac{2-\beta}{2}}\left(\int^s_0\|\uu(\tau)\|_2^{\frac{4\alpha}{2-\beta}}\,{\rm{d}}\tau\right)^{\frac{2-\beta}{2}}\,{\rm{d}}\xi\ds\\
&\hspace{5mm} +\int^t_0(1+s)^{3+\gamma}\int_{L(s)}|\xi|^2s^2\,{\rm{d}}\xi\ds\\
\end{align*}
Now we shall use Theorem \ref{prev_thm}. Then from the above inequality, we have
\begin{align*}
\int^t_0(  &  1+s)^{3+\gamma}\int_{L(s)}|\xi|^2\left(1+\left(\int^s_0\|\uuu(\tau)\|_2^{\frac{2\alpha}{2-\beta}}{\rm{d}}\tau\right)^{\frac{2-\beta}{2}}+ \left(\int^s_0\|\uu(\tau)\|_2^{\frac{2\alpha}{2-\beta}}{\rm{d}}\tau\right)^{\frac{2-\beta}{2}}+\int^s_0\|\ww(\tau)\|^2_2{\rm{d}}\tau\right)^2{\rm{d}}\xi\ds\\
&\lesssim \int^t_0(1+s)^{\frac{1}{2}+\gamma}\ds +\int^t_0(1+s)^{\frac{5}{2}+\gamma-\frac{3}{2}\alpha-\beta}\ds+\int^t_0(1+s)^{\frac{5}{2}+\gamma}\ds\\
&\lesssim(1+t)^{\frac{3}{2}+\gamma}+(1+t)^{\frac{7}{2}+\gamma-\frac{3}{2}\alpha-\beta}+(1+t)^{\frac{7}{2}+\gamma}.
\end{align*}

Note that $\frac{7}{2}-\frac{3}{2}\alpha-\beta\leq\frac{3}{2}$ as $p^-\geq\frac{17}{7}$. Therefore, as before, we obtain
\begin{align*}
(1+t)^{4+\gamma}\|\ww(t)\|^2_2
& \lesssim 1+(1+t)^{4}+(1+t)^{\frac{3}{2}+\gamma}+(1+t)^{\frac{7}{2}+\gamma-\frac{3}{2}\alpha-\beta}\\
&\hspace{5mm}+(1+t)^{\frac{7}{2}+\gamma}+\int^t_0(1+s)^{4+\gamma}g(\uu,\ww)\ds\\
& \lesssim (1+t)^{\frac{7}{2}+\gamma}+\int^t_0(1+s)^{4+\gamma}(\|\nabla\uu(s)\|^2_2+\|\nabla^2\uu(s)\|^2_2)\|\ww(s)\|^2_2\ds.
\end{align*}
Then by the use of Gronwall's inequality, we deduce that
\begin{align*}
(1+t)^{4+\gamma}\|\ww(t)\|^2_2
& \lesssim (1+t)^{\frac{7}{2}+\gamma}\exp\left(\int^{\infty}_0(\|\nabla\uu(s)\|_2^2+\|\nabla^2\uu(s)\|^2_2)\ds\right)\\
& \lesssim (1+t)^{\frac{7}{2}+\gamma}\exp\left(\int^{\infty}_0(\mathcal{I}_p(\uu)+\mathcal{J}_p(\uu))\ds\right),
\end{align*}
which lead us to by \eqref{energy_1} and \eqref{energy_2}, 
\begin{equation}\label{new_iter_1}
\|\ww(t)\|_2\lesssim(1+t)^{-\frac{1}{4}},\quad\forall t>0.
\end{equation}
We proceed similarly as we did for the case $p^-\geq 3$.  If we substitute \eqref{new_iter_1} into \eqref{repeat_2}, then we have
\begin{align*}
(1+t)^{4+\gamma}\|\ww(t)\|^2_2
& \lesssim 1+(1+t)^{4}+(1+t)^{\frac{7}{2}+\gamma-\frac{3}{2}\alpha-\beta}+(1+t)^{3+\gamma}+\int^t_0(1+s)^{4+\gamma}g(\uu,\ww)\ds\\
& \lesssim (1+t)^{3+\gamma}+\int^t_0(1+s)^{4+\gamma}(\|\nabla\uu(s)\|^2_2+\|\nabla^2\uu(s)\|^2_2)\|\ww(s)\|^2_2\ds,
\end{align*}
where we have used the fact that $\frac{7}{2}-\frac{3}{2}\alpha-\beta<3$.
Consequently we deduce that by Gronwall's inequality
\begin{equation}\label{new_iter_2}
\|\ww(t)\|_2\lesssim(1+t)^{-\frac{1}{2}},\quad\forall t>0.
\end{equation}

Likewise, if we put \eqref{new_iter_2} into \eqref{repeat_2}, we have due to the fact that $\frac{7}{2}-\frac{3}{2}\alpha-\beta<\frac{5}{2}$,
\begin{align*}
(1+t)^{4+\gamma}\|\ww(t)\|^2_2
& \lesssim (1+t)^{\frac{5}{2}+\gamma}+\int^t_0(1+s)^{4+\gamma}(\|\nabla\uu(s)\|^2_2+\|\nabla^2\uu(s)\|^2_2)\|\ww(s)\|^2_2\ds,
\end{align*}
and this gives us
\begin{equation}\label{new_iter_3}
\|\ww(t)\|_2\lesssim(1+t)^{-\frac{3}{4}},\quad\forall t>0.
\end{equation}
In the same way, again from the fact that $\frac{7}{2}-\frac{3}{2}\alpha-\beta<2$. we have that
\begin{align*}
(1+t)^{4+\gamma}\|\ww(t)\|^2_2
& \lesssim (1+t)^{2+\gamma}+\int^t_0(1+s)^{4+\gamma}(\|\nabla\uu(s)\|^2_2+\|\nabla^2\uu(s)\|^2_2)\|\ww(s)\|^2_2\ds,
\end{align*}
and
\begin{equation}\label{new_iter_4}
\|\ww(t)\|_2\lesssim(1+t)^{-1},\quad\forall t>0.
\end{equation}

Finally, repeating the same process once more leads us to
\begin{align*}
(1+t)^{4+\gamma}\|\ww(t)\|^2_2
& \lesssim 1+(1+t)^{4}+(1+t)^{\frac{7}{2}+\gamma-\frac{3}{2}\alpha-\beta}+(1+t)^{4}+\int^t_0(1+s)^{4+\gamma}g(\uu,\ww)\ds\\
& \lesssim (1+t)^{4}+\int^t_0(1+s)^{4+\gamma}(\|\nabla\uu(s)\|^2_2+\|\nabla^2\uu(s)\|^2_2)\|\ww(s)\|^2_2\ds,
\end{align*}
where we have used the fact that $\gamma<\frac{5}{2}$ and $\log(1+x)\leq C_{\alpha} x^{\alpha}$ for arbitrarily small $\alpha>0$. Note that this is exactly the point where $p^-\geq\frac{17}{7}$ is needed instead of $p^-\geq\frac{11}{5}$. 
Finally we conclude that the following decay estimate holds
\begin{equation}\label{new_iter_final}
\|\uuu(t)-\uu(t)\|_2=\|\ww(t)\|_2\lesssim(1+t)^{-\frac{\gamma}{2}},\quad\forall t>0.
\end{equation}

\end{section}

\begin{section}{Estimate for lower bounds}
In this section, we shall derive the explicit lower bound of convergence rate to the problem under consideration. Let us first denote the difference $\ppsi(t,x)=\ww(t,x)-\pphi(t,x)$ where $\ww(t,x)$ is the solution of the problem \eqref{d_eq1}-\eqref{d_initial} and $\pphi=e^{\Delta t}\ww_0$ be the solution of linear heat equations discussed in Lemma \ref{linear_decay}. Then the difference $\ppsi(t,x)$ satisfies the following system of PDEs in the distributional sense:
\begin{align}
\partial_t\ppsi+(\uuu\cdot\nabla)\ww+(\ww\cdot\nabla)\uu-\Delta\ppsi+\nabla r&={\rm{div}}\,\GGG(\DD\uuu)-{\rm{div}}\,\GGG(\DD\uu)\label{pot_eq1}\\
{\rm{div}}\,\ppsi&=0, \label{pot_eq2}\\
\ppsi(0,x)&=0.\label{pot_initial}
\end{align}
Note that since the heat equation preserves the divergence-free condition, we have \eqref{pot_eq2}. We start with the following lemma on the estimate of $|\widehat{\ppsi}|$. 
%Since the proof of the lemma is almost same as the proof of Lemma \ref{second_est}, we omit the proof.

\begin{lemma}\label{pot_second_est}
For the solution $\ppsi(t,x)$ defined in \eqref{pot_eq1}-\eqref{pot_initial}, we have the following estimates:
\begin{equation}\label{pot_repeat_1}
|\widehat{\ppsi}(t,\xi)|\lesssim|\xi|+|\xi|\int^t_0\|\ww(s)\|_2\ds.
\end{equation}
\end{lemma}
\begin{proof}
The proof is almost identical to the proof of Lemma \ref{second_est}. Hence the proof is done for the case $p^-\geq3$ as we have $\ppsi(0,x)=0$. For the case of $\frac{17}{7}\leq p^- <3$, we shall use Theorem \ref{prev_thm}, and then we have
\begin{align*}
\left(\int^t_0\|\uuu(s)\|^{\frac{2\alpha}{2-\beta}}\ds\right)^{\frac{2-\beta}{2}}+\left(\int^t_0\|\uu(s)\|^{\frac{2\alpha}{2-\beta}}\ds\right)^{\frac{2-\beta}{2}}
& \lesssim \left(\int^t_0(1+s)^{-\frac{3\alpha}{2(2-\beta)}}\right)^{\frac{2-\beta}{2}}\\
& \lesssim (1+t)^{\left(-\frac{3\alpha}{2(2-\beta)}+1\right)\left(\frac{2-\beta}{2}\right)}\\
& \lesssim (1+t)^{-\frac{3}{4}\alpha+\frac{2-\beta}{2}}\\
& \lesssim (1+t)^{\frac{17-7p^-}{16}}\leq C,
\end{align*}
where we have used the condition $\frac{17}{7}\leq p^-$ in the last inequality. Therefore, we have completed the proof.

\end{proof}

We shall now estimate the decay rate of weak solutions $\ppsi(t,x)$ and compare it with the decay rate of $\ww(t,x)$. If take $L^2$-inner product of \eqref{pot_eq1} with $\ppsi(t,x)$, by Korn's inequality, we have
\begin{align*}
\frac{1}{2}\frac{\rm{d}}{\dt}&\|\ppsi\|^2_2+\|\nabla\ppsi\|^2_2+\int_{\R^3}(\uuu\cdot\nabla)\ww\cdot\ppsi\dx+\int_{\R^3}(\ww\cdot\nabla)\uu\cdot\ppsi\dx\\
&=\int_{\R^3}\left[\GGG(\DD\uu):\DD\ppsi-\GGG(\DD\uuu):\DD\ppsi\right]\dx.
\end{align*}
Note first that by Lemma \ref{G_mon}, we obtain
\begin{align*}
&\hspace{5mm}\int_{\R^3}\left[\GGG(\DD\uu):\DD\ppsi-\GGG(\DD\uuu):\DD\ppsi\right]\dx\\
&=-\int_{\R^3}\left(\GGG(\DD\uuu)-\GGG(\DD\uu)\right):(\DD\uuu-\DD\uu)\dx+\int_{\R^3}\left(\GGG(\DD\uuu)-\GGG(\DD\uu)\right):\DD\pphi\dx\\
&\leq \int_{\R^3}\left(\GGG(\DD\uuu)-\GGG(\DD\uu)\right):\DD\pphi\dx,
\end{align*}
which leads us to the inequality
\begin{align*}
\frac{\rm{d}}{\dt}\|\ppsi\|^2_2+\|\nabla\ppsi\|^2_2 
&\lesssim\int_{\R^3}\left(\GGG(\DD\uuu)-\GGG(\DD\uu)\right):\DD\pphi\dx\\
& \hspace{3mm}-\int_{\R^3}(\uuu\cdot\nabla)\ww\cdot\ppsi\dx-\int_{\R^3}(\ww\cdot\nabla)\uu\cdot\ppsi\dx.
\end{align*}
For the first term on the right-hand side, by Gagaliardo--Nirenberg interpolation inequality (see, for example, \cite{GNI}),
\begin{align*}
 \int_{\R^3}\left(\GGG(\DD\uuu)-\GGG(\DD\uu)\right):\DD\pphi\dx
 &\lesssim\|\GGG(\DD\uuu)-\GGG(\DD\uu)\|_1\|\DD\pphi\|_{\infty}\\
 &\lesssim \|\GGG(\DD\uuu)-\GGG(\DD\uu)\|_1\|\pphi\|^{\frac{1}{6}}_2\|\nabla^3\pphi\|^{\frac{5}{6}}_2.
\end{align*}

For the second term, by H\"older's inequality, Young's inequality, Sobolev embedding and the skew symmetry of the convective term, we have for some small constants $\varepsilon>0$ and $\delta>0$,
\begin{align*}
-\int_{\R^3}(\uuu\cdot\nabla)\ww\cdot\ppsi\dx
&=\int_{\R^3}(\uuu\cdot\nabla)\ppsi\cdot\ww\dx=\int_{\R^3}(\uuu\cdot\nabla)\ppsi\cdot(\ppsi+\pphi)\dx\\
&\leq C\|\uuu\|_2\|\nabla\ppsi\|_2\|\pphi\|_{\infty}\leq\varepsilon\|\nabla\ppsi\|^2_2+C(\varepsilon)\|\uuu\|^2_2\|\nabla^{\frac{3}{2}+\delta}\pphi\|^2_2.
\end{align*}
Finally for the third term, note that
\[-\int_{\R^3}(\ww\cdot\nabla)\uu\cdot\ppsi\dx=\int_{\R^3}(\ww\cdot\nabla)\ppsi\cdot\uu\dx=\int_{\R^3}(\pphi\cdot\nabla)\ppsi\cdot\uu\dx+\int_{\R^3}(\ppsi\cdot\nabla)\ppsi\cdot\uu\dx\eqdef{\rm{I}}+{\rm{II}}.\]
By H\"older's inequality, Young's inequality, Sobolev embedding and the interpolation inequality, we obtain  for some small constants $\varepsilon>0$ and $\delta>0$,
\[{\rm{I}}\leq\|\pphi\|_{\infty}\|\nabla\ppsi\|_2\|\uu\|_2\leq\varepsilon\|\nabla\ppsi\|^2_2+C(\varepsilon)\|\pphi\|^2_{\infty}\|\uu\|^2_2\leq
\varepsilon\|\nabla\ppsi\|^2_2+C(\varepsilon)\|\uu\|^2_2\|\nabla^{\frac{3}{2}+\delta}\pphi\|^2_2,\]
and
\begin{align*}
{\rm{II}}=-\int_{\R^3}(\ppsi\cdot\nabla)\uu\cdot\ppsi\dx
&\leq\|\ppsi\|_2\|\nabla\uu\|_3\|\ppsi\|_6\\
&\leq C\|\ppsi\|_2\|\nabla\uu\|^{\frac{1}{2}}_2\|\nabla^2\uu\|^{\frac{1}{2}}_2\|\nabla\ppsi\|_2\\
&\leq C(\varepsilon)\|\ppsi\|^2_2\|\nabla\uu\|_2\|\nabla^2\uu\|_2+\varepsilon\|\nabla\ppsi\|^2_2\\
&\leq C\left(\|\nabla\uu\|^2_2+\|\nabla^2\uu\|^2_2\right)\|\ppsi\|^2_2+\varepsilon\|\nabla\ppsi\|^2_2.
\end{align*}
By combining above inequalities and using Lemma \ref{linear_decay} and the fact $\sup_{t>0}(\|\uuu(t)\|^2_2+\|\uu(t)\|^2_2)\leq C$, we have
\begin{align*}
\frac{\rm{d}}{\dt}\|\ppsi\|^2_2+\|\nabla\ppsi\|^2_2
& \lesssim\|\GGG(\DD\uuu)-\GGG(\DD\uu)\|_1\|\pphi\|^{\frac{1}{6}}_2\|\nabla^3\pphi\|^{\frac{5}{6}}_2\\
&\hspace{3mm}+\|\nabla^{\frac{3}{2}+\delta}\pphi\|^2_2+(\|\nabla\uu\|^2_2+\|\nabla^2\uu\|^2_2)\|\ppsi\|^2_2\\
& \lesssim (1+t)^{-\frac{5+2\gamma}{4}}\|\GGG(\DD\uuu)-\GGG(\DD\uu)\|_1\\
&\hspace{3mm}+f(t)(1+t)^{-\left(\frac{3}{2}+\delta+\gamma\right)}+(\|\nabla\uu\|^2_2+\|\nabla^2\uu\|^2_2)\|\ppsi\|^2_2.
\end{align*}
Let us denote $h(\uu,\ppsi)=\left(\|\nabla\uu\|^2_2+\|\nabla^2\uu\|^2_2\right)\|\ppsi\|^2_2$ and assume again that $f(t)$ is a smooth function satisfying $f(0)=1$, $f(t)>0$ and $f'(t)>0$. Then similarly as we did in \eqref{d_energy}--\eqref{mid_ineq_const}, for some constant $C_0>0$, we deduce that
\begin{align*}
&\hspace{5mm}\frac{\rm{d}}{\dt}\left(f(t)\int_{\R^3}|\widehat{\ppsi}(t,\xi)|^2\dxi\right)+C_0f(t)\int_{\R^3}|\xi|^2|\widehat{\ppsi}(t,\xi)|^2\dxi\\
&\leq f'(t)\int_{\R^3}|\widehat{\ppsi}(t,\xi)|^2\dxi+f(t) (1+t)^{-\frac{5+2\gamma}{4}}\|\GGG(\DD\uuu)-\GGG(\DD\uu)\|_1\\
&\hspace{3mm}+f(t)(1+t)^{-\left(\frac{3}{2}+\delta+\gamma\right)}+f(t)h(\uu,\ppsi).
\end{align*}
We shall define the set $M(t)\defeq\{\xi\in\R^3:C_0|\xi|^2f(t)\leq f'(t)\}$ where $C_0>0$ is the constant appearing in the above inequality. Then as we did before, we have 
\begin{align*}
&\hspace{5mm}\frac{\rm{d}}{\dt}\left(f(t)\int_{\R^3}|\widehat{\ppsi}(t,\xi)|^2\dxi\right)\\
&\leq f'(t)\int_{M(t)}|\widehat{\ppsi}(t,\xi)|^2\dxi+f(t) (1+t)^{-\frac{5+2\gamma}{4}}\|\GGG(\DD\uuu)-\GGG(\DD\uu)\|_1\\
&\hspace{3mm}+f(t)(1+t)^{-\left(\frac{3}{2}+\delta+\gamma\right)}+f(t)h(\uu,\ppsi).
\end{align*}
Now, setting $f(t)=(1+t)^{4+\gamma}$ and integrating over $(0,t)$ yields
\begin{align*}
(1+t)^{4+\gamma}\|\ppsi\|^2_2&\lesssim \int^t_0(1+s)^{3+\gamma}\int_{M(s)}|\widehat{\ppsi}(\xi)|^2\dxi\ds+(1+t)^{\frac{11+2\gamma}{4}}\\
&\hspace{3mm} +\int^t_0(1+s)^{\frac{5}{2}}\ds+\int^t_0(1+s)^{4+\gamma}h(\uu,\ppsi)\ds,
\end{align*}
where we have used Lemma \ref{main_est} and Theorem \ref{prev_thm} with the condition $\frac{17}{7}\leq p^-$ regarding the integrability of $\|\GGG(\DD\uuu)-\GGG(\DD\uu)\|_1$ in time. Now, by Lemma \ref{pot_second_est}, \eqref{iter_final} and the fact $2<\gamma<\frac{5}{2}$, we have
\begin{align*}
(1+t)^{4+\gamma}\|\ppsi\|^2_2
&\lesssim\int^t_0(1+s)^{3+\gamma}\int_{M(s)}|\xi|^2\dxi\ds+\int^t_0(1+s)^{3+\gamma}\int_{M(s)}|\xi|^2\left(\int^s_0\|\ww(\tau)\|\,{\rm{d}}\tau\right)^2\dxi\ds\\
&\hspace{5mm}+(1+t)^{\frac{11+2\gamma}{4}}+(1+t)^{\frac{7}{2}}+\int^t_0(1+s)^{4+\gamma}h(\uu,\ppsi)\ds\\
&\lesssim\int^t_0(1+s)^{3+\gamma}\int_{M(s)}|\xi|^2\dxi\ds+\int^t_0(1+s)^{3+\gamma}\int_{M(s)}|\xi|^2\left(\int^s_0(1+\tau)^{-\frac{\gamma}{2}}\,{\rm{d}}\tau\right)^2\dxi\ds\\
&\hspace{5mm}+(1+t)^{\frac{11+2\gamma}{4}}+(1+t)^{\frac{7}{2}}+\int^t_0(1+s)^{4+\gamma}h(\uu,\ppsi)\ds\\
&\lesssim \int^t_0(1+s)^{\frac{1}{2}+\gamma}\ds+\int^t_0(1+s)^{\frac{5}{2}}\ds+(1+t)^{\frac{11+2\gamma}{4}}+(1+t)^{\frac{7}{2}}+\int^t_0(1+s)^{4+\gamma}h(\uu,\ppsi)\ds\\
&\lesssim (1+s)^{\frac{3}{2}+\gamma}+(1+s)^{\frac{7}{2}}+(1+t)^{\frac{11+2\gamma}{4}}+(1+t)^{\frac{7}{2}}+\int^t_0(1+s)^{4+\gamma}h(\uu,\ppsi)\ds\\
&\lesssim (1+t)^{\frac{11+2\gamma}{4}}+\int^t_0(1+s)^{4+\gamma}(\|\nabla\uu\|^2_2+\|\nabla^2\uu\|^2_2)\|\ppsi\|^2_2.
\end{align*}
By Gronwall's inequality, we have
\begin{align*}
(1+t)^{4+\gamma}\|\ppsi\|^2_2
&\lesssim(1+t)^{\frac{11+2\gamma}{4}}\exp\left(\int^{\infty}_0\left(\|\nabla\uu(s)\|^2_2+\|\nabla^2\uu(s)\|^2_2\right)\dx\right)\\
&\lesssim(1+t)^{\frac{11+2\gamma}{4}}\exp\left(\int^{\infty}_0\left(\mathcal{I}_p(\uu)+\mathcal{J}_p(\uu)\right)\dx\right),
\end{align*}
which leads us to the following inequality:
\[\|\ppsi\|^2_2\lesssim (1+t)^{-\frac{5+2\gamma}{4}}.\]
Finally, by Lemma \ref{linear_decay}, we conclude for sufficiently large $t>0$ that
\begin{align*}
\|\uuu(t)-\uu(t)\|^2_2
&=\|\ww(t)\|^2_2=\|\ppsi(t)+\pphi(t)\|^2_2\geq\|\pphi(t)\|^2_2-\|\ppsi(t)\|^2_2\\
&\gtrsim(1+t)^{-\gamma}-(1+t)^{-\frac{5+2\gamma}{4}}\gtrsim(1+t)^{-\gamma},
\end{align*}
which is the desired estimate for the lower bound of convergence rate.
\end{section}

\begin{section}{Conclusion}
In this paper, we have concerned with the upper and lower bounds of convergence rates for the strong solutions of the $3$D generalized Newtonian fluids with variable power-law index. This mathematical model describes the rheological behaviour of an incompressible electrorheological fluid. In specific, we have obtained the following decay estimates: if the initial perturbation satisfies \eqref{initial_decay} with $2<\gamma<\frac{5}{2}$, we have
\[(1+t)^{-\frac{\gamma}{2}}\lesssim \|\uuu(t)-\uu(t)\|_2\lesssim(1+t)^{-\frac{\gamma}{2}},\]
for sufficiently large $t>0$. These estimates can be regarded as optimal in the sense that they coincide with those of the linear heat equations. The ideas of the proof include the comparison of the convergence rate for the given model with the one of heat equations for $t>0$ sufficiently large, together with the Fourier splitting methods and iterative process. In order to prove the main estimates we assumed that $p^-\geq\frac{17}{7}$ which was required to ensure that the itereation process reached to the optimal rate. Note, however, we still have a possibility to relax this condition and obtain the similar result with the mild restriction $p^-\geq\frac{11}{5}$, as presented in \cite{Ko2022} where optimal algebraic decay rate of the strong solutions of the model under consideration was obtained provided that $p^-\geq\frac{11}{5}$. We leave it as further consideration.

In fact, the argument used in the present paper can improve the existing results for the model with a constant-power-law index. For example, in \cite{perturb_non}, the authors established the same result for the bipolar generalized Navier--Stokes equations with fixed power-law index provided that $p\geq 3$. If we apply the same method presented in this paper for the case of $\frac{17}{7}\leq p<3$, it is straightforward that we can extend the result in \cite{perturb_non} to the case of $p\geq\frac{17}{7}$.

Another interesting future research direction is to derive the similar bounds of convergence rates for the derivative of the strong solutions. In \cite{Ko2022}, the optimal $L^2$ decay rate for the gradient of the strong solutions was also obtained. By employing suitable modifications of the arguments used in \cite{Ko2022}, we expect to obtain the following upper and lower bounds of convergence rates:
\[(1+t)^{-\frac{1+\gamma}{2}}\lesssim \|\nabla\uuu(t)-\nabla\uu(t)\|_2\lesssim(1+t)^{-\frac{1+\gamma}{2}},\]
which will be addressed in the forthcoming paper.
\end{section}

%\section*{Acknowledgements}
%Seungchan Ko's work was supported by the UK Engineering and Physical Sciences Research Council [EP/L015811/1].

\bibliography{references}

\def\ocirc#1{\ifmmode\setbox0=\hbox{$#1$}\dimen0=\ht0 \advance\dimen0
  by1pt\rlap{\hbox to\wd0{\hss\raise\dimen0
  \hbox{\hskip.2em$\scriptscriptstyle\circ$}\hss}}#1\else {\accent"17 #1}\fi}
\begin{thebibliography}{10}

\bibitem{Bae_1999}
H.-O. Bae.
\newblock Existence, regularity, and decay rate of solutions of non-{N}ewtonian
  flow.
\newblock {\em J. Math. Anal. Appl.}, 231(2):467--491, 1999.

\bibitem{Chen_1991}
Z.~M. Chen.
\newblock A sharp decay result on strong solutions of the {N}avier-{S}tokes
  equations in the whole space.
\newblock {\em Comm. Partial Differential Equations}, 16(4-5):801--820, 1991.

\bibitem{DHHR2011}
L.~Diening, P.~Harjulehto, P.~H{\"a}st{\"o}, and
  M.~R{\ocirc{u}}{\v{z}}i{\v{c}}ka.
\newblock {\em Lebesgue and {S}obolev spaces with variable exponents}, volume
  2017 of {\em Lecture Notes in Mathematics}.
\newblock Springer, Heidelberg, 2011.

\bibitem{strong_sol_per}
L.~Diening and M.~R\ocirc{u}\v{z}i\v{c}ka.
\newblock Strong solutions for generalized {N}ewtonian fluids.
\newblock {\em J. Math. Fluid Mech.}, 7(3):413--450, 2005.

\bibitem{Dong_2004}
B.~Dong and Y.~Li.
\newblock Large time behavior to the system of incompressible non-{N}ewtonian
  fluids in {${\bf R}^2$}.
\newblock {\em J. Math. Anal. Appl.}, 298(2):667--676, 2004.

\bibitem{Dong}
B.-Q. Dong.
\newblock Decay of solutions to equations modelling incompressible bipolar
  non-{N}ewtonian fluids.
\newblock {\em Electron. J. Differential Equations}, pages No. 125, 13, 2005.

\bibitem{non_1}
B.~Guo and P.~Zhu.
\newblock Algebraic {$L^2$} decay for the solution to a class system of
  non-{N}ewtonian fluid in {$\bold R^n$}.
\newblock {\em J. Math. Phys.}, 41(1):349--356, 2000.

\bibitem{Han_2015}
P.~Han.
\newblock Decay results of higher-order norms for the {N}avier-{S}tokes flows
  in 3{D} exterior domains.
\newblock {\em Comm. Math. Phys.}, 334(1):397--432, 2015.

\bibitem{Han_2016_1}
P.~Han.
\newblock Large time behavior for the nonstationary {N}avier-{S}tokes flows in
  the half-space.
\newblock {\em Adv. Math.}, 288:1--58, 2016.

\bibitem{Han_2016_2}
P.~Han.
\newblock Long-time behavior for {N}avier-{S}tokes flows in a two-dimensional
  exterior domain.
\newblock {\em J. Funct. Anal.}, 270(3):1091--1152, 2016.

\bibitem{GNI}
D.~Henry.
\newblock {\em Geometric theory of semilinear parabolic equations}, volume 840
  of {\em Lecture Notes in Mathematics}.
\newblock Springer-Verlag, Berlin-New York, 1981.

\bibitem{perturb_NS}
Y.~Jia, Q.~Xie, and W.~Wang.
\newblock The optimal upper and lower bounds of convergence rates for the 3{D}
  {N}avier-{S}tokes equations under large initial perturbation.
\newblock {\em J. Math. Anal. Appl.}, 459(1):437--452, 2018.

\bibitem{ref_1}
R.~Kajikiya and T.~Miyakawa.
\newblock On {$L^2$} decay of weak solutions of the {N}avier-{S}tokes equations
  in {${\bf R}^n$}.
\newblock {\em Math. Z.}, 192(1):135--148, 1986.

\bibitem{kato}
T.~Kato.
\newblock Strong {$L^{p}$}-solutions of the {N}avier-{S}tokes equation in
  {${\bf R}^{m}$}, with applications to weak solutions.
\newblock {\em Math. Z.}, 187(4):471--480, 1984.

\bibitem{Ko2022}
S.~Ko.
\newblock Temporal decay of strong solutions for generalized newtonian fluids
  with variable power-law index.
\newblock {\em arXiv:2203.10700 [math.AP]}, 2022.

\bibitem{Malek}
J.~M\'{a}lek, J.~Ne\v{c}as, M.~Rokyta, and M.~R\ocirc{u}\v{z}i\v{c}ka.
\newblock {\em Weak and measure-valued solutions to evolutionary {PDE}s},
  volume~13 of {\em Applied Mathematics and Mathematical Computation}.
\newblock Chapman \& Hall, London, 1996.

\bibitem{non_2}
{\v{S}}.~Ne\v{c}asov\'{a} and P.~Penel.
\newblock {$L^2$} decay for weak solution to equations of non-{N}ewtonian
  incompressible fluids in the whole space.
\newblock In {\em Proceedings of the {T}hird {W}orld {C}ongress of {N}onlinear
  {A}nalysts, {P}art 6 ({C}atania, 2000)}, volume~47, pages 4181--4192, 2001.

\bibitem{titi}
M.~Oliver and E.~S. Titi.
\newblock Remark on the rate of decay of higher order derivatives for solutions
  to the {N}avier-{S}tokes equations in {${\bf R}^n$}.
\newblock {\em J. Funct. Anal.}, 172(1):1--18, 2000.

\bibitem{Cauchy}
M.~Pokorn\'{y}.
\newblock Cauchy problem for the non-{N}ewtonian viscous incompressible fluid.
\newblock {\em Appl. Math.}, 41(3):169--201, 1996.

\bibitem{R2000}
M.~R{\ocirc{u}}{\v{z}}i{\v{c}}ka.
\newblock {\em Electrorheological fluids: modeling and mathematical theory},
  volume 1748 of {\em Lecture Notes in Mathematics}.
\newblock Springer-Verlag, Berlin, 2000.

\bibitem{R2004}
M.~R{\ocirc{u}}{\v{z}}i{\v{c}}ka.
\newblock Modeling, mathematical and numerical analysis of electrorheological
  fluids.
\newblock {\em Appl. Math.}, 49(6):565--609, 2004.

\bibitem{Schonbek_1}
M.~E. Schonbek.
\newblock {$L^2$} decay for weak solutions of the {N}avier-{S}tokes equations.
\newblock {\em Arch. Rational Mech. Anal.}, 88(3):209--222, 1985.

\bibitem{Schonbek_2}
M.~E. Schonbek.
\newblock Large time behaviour of solutions to the {N}avier-{S}tokes equations.
\newblock {\em Comm. Partial Differential Equations}, 11(7):733--763, 1986.

\bibitem{Schonbek_3}
M.~E. Schonbek.
\newblock Lower bounds of rates of decay for solutions to the {N}avier-{S}tokes
  equations.
\newblock {\em J. Amer. Math. Soc.}, 4(3):423--449, 1991.

\bibitem{Schonbek_4}
M.~E. Schonbek.
\newblock Asymptotic behavior of solutions to the three-dimensional
  {N}avier-{S}tokes equations.
\newblock {\em Indiana Univ. Math. J.}, 41(3):809--823, 1992.

\bibitem{ref_2}
M.~Wiegner.
\newblock Decay results for weak solutions of the {N}avier-{S}tokes equations
  on {${\bf R}^n$}.
\newblock {\em J. London Math. Soc. (2)}, 35(2):303--313, 1987.

\bibitem{perturb_non}
Q.~Xie, Y.~Guo, and B.-Q. Dong.
\newblock Upper and lower convergence rates for weak solutions of the 3{D}
  non-{N}ewtonian flows.
\newblock {\em J. Math. Anal. Appl.}, 494(2):Paper No. 124641, 21, 2021.

\bibitem{Zhang_1995}
L.~H. Zhang.
\newblock Sharp rate of decay of solutions to {$2$}-dimensional
  {N}avier-{S}tokes equations.
\newblock {\em Comm. Partial Differential Equations}, 20(1-2):119--127, 1995.

\end{thebibliography}
\bibliographystyle{abbrv}

%\section*{Appendix. Proof of Theorem \ref{thm:simple_system} and \ref{thm:geometric_length}}

\end{document}